\documentclass{amsart}
\usepackage{mathptmx, bbm, amscd, amssymb, enumerate, colonequals, mathdots, xcolor, tikz-cd,mathtools, amsthm}
\usepackage{nicematrix1}
\usepackage{color}
\definecolor{chianti}{rgb}{0.6,0,0}
\definecolor{meretale}{rgb}{0,0,.6}
\definecolor{leaf}{rgb}{0,.35,0}
\usepackage[colorlinks=true, pagebackref, hyperindex, citecolor=meretale, urlcolor=leaf, linkcolor=chianti]{hyperref}

\DeclareMathOperator{\height}{ht}

\DeclareMathOperator{\ann}{ann}

\DeclareMathOperator{\Hom}{Hom}

\DeclareMathOperator{\initial}{in}
\DeclareMathOperator{\Ass}{Ass}

\usepackage{thmtools}
\usepackage{thm-restate}

\newtheorem*{thm*}{Theorem}

\newtheorem{thm}{Theorem}[section]
\newtheorem{lem}[thm]{Lemma}
\newtheorem{cor}[thm]{Corollary}
\newtheorem{prop}[thm]{Proposition}
\newtheorem{rem}[thm]{Remark}

\theoremstyle{definition}
\newtheorem{defn}[thm]{Definition}

\begin{document}
\title{Linkage and $F$-Regularity of determinantal rings}

\author{Vaibhav Pandey}
\address{Department of Mathematics, Purdue University, 150 N University St., West Lafayette, IN~47907, USA}
\email{pandey94@purdue.edu}

\author{Yevgeniya Tarasova}
\address{Department of Mathematics, Purdue University, 150 N University St., West Lafayette, IN~47907, USA}
\email{ytarasov@purdue.edu}


\begin{abstract}
In this paper, we prove that the generic link of a generic determinantal ring defined by maximal minors is strongly $F$-regular. In the process, we strengthen a result of Chardin and Ulrich in the graded setting. They showed that the generic residual intersections of a complete intersection ring with rational singularities again have rational singularities. We show that if the said complete intersection is defined by homogeneous elements and is $F$-rational, then in fact, its generic residual intersections are strongly $F$-regular in positive prime characteristic.

Hochster and Huneke showed that determinantal rings are strongly $F$-regular; however, their proof is quite involved. Our techniques allow us to give a new and simple proof of the strong $F$-regularity of determinantal rings defined by maximal minors. 
\end{abstract}

\maketitle

\section{Introduction}

The modern study of linkage, or liaison theory, started with the work of Peskine and Szpiro in \cite{PS} and has been developed extensively by Huneke and Ulrich in \cite{HU1}, \cite{HU2}, and \cite{HU3}. Given proper ideals $I$ and $J$ in a polynomial ring $R$ over a field, we say that $R/I$ and $R/J$ are \emph{linked} if there exists a regular sequence $\mathfrak{a}$ in $R$ such that \[I = \mathfrak{a} : J \qquad \text{and} \qquad J = \mathfrak{a} : I. \] If the ideals $I$ and $J$ do not share any irreducible component, $R/J$ is said to be a \emph{geometric link} of $R/I$. In other words, the set theoretic union of the varieties defined by $I$ and $J$ (in the affine or projective space) is a complete intersection. When this complete intersection is chosen in the most general manner, that is, when the generators of $\mathfrak{a}$ are taken to be generic combinations of the generators of $I$, then $R/J$ is said to be the \emph{generic link} of $R/I$. Since the generic link is a deformation of any other link, homological properties of any link can be traced from those of generic links (see \cite[Proposition 2.14]{HU2}). However, given a singularity on the variety defined by $I$, not much is known about the singularities of its generic link despite the rapid development in the theory of singularities of algebraic varieties over the past few decades. In this paper, we prove


\begin{thm*} \label{Theorem A} (Theorem \ref{theoremGenericSFR})
     Let $X = (x_{i,j})$ be a $t \times n$ matrix of indeterminates for $n\geq t$, $K$ a field, and $R = K[X]$. Let $I_t(X)$ denote the ideal of $R$ generated by the size $t$ minors of $X$. Then, 
     \begin{enumerate}
         \item If $K$ has characteristic zero, the generic link of $R/I_t(X)$ has rational singularities.
         
         \item If $K$ is an $F$-finite field of positive characteristic, the generic link of $R/I_t(X)$ is strongly $F$-regular.
     \end{enumerate}
\end{thm*}

The above theorem readily implies that if $I$ is an ideal of maximal minors which has generic height, then \emph{any} link of $I$ has a deformation which is strongly $F$-regular.


Major progress towards understanding the behavior of singularities under linkage was made by Chardin and Ulrich in \cite[Proposition 3.4]{CU}. They proved that if $R/I$ is a complete intersection ring with rational singularities, then the generic residual intersections of $R/I$ (a generalization of generic links) also have rational singularities. This result has been vastly extended in characteristic zero in \cite{Niu} (see also \cite{MPGTZ}). We prove that if the ideal $I$ is generated by a regular sequence of homogeneous elements and $R/I$ is $F$-rational then the generic residual intersections of $R/I$ are, in fact, strongly $F$-regular. 

\begin{thm*}\label{Theorem B}(Theorem \ref{theoremGenCU})
     Let $R = K[x_1, \ldots, x_n]$ be a polynomial ring over a field $K$ and $I$ be an ideal of $R$ generated by a regular sequence of homogeneous elements. Then,
     \begin{enumerate}
         \item If $K$ has characteristic zero and $R/I$ has rational singularities, then the generic residual intersections of $R/I$ also have rational singularities.
         
         \item If $K$ is an $F$-finite field of positive characteristic and $R/I$ is $F$-rational, then the generic residual intersections of $R/I$ are strongly $F$-regular.
     \end{enumerate}
\end{thm*}

The bulk of the work for the above result is done in the case where $I$ is the homogeneous maximal ideal of $R$ (see Theorem \ref{theoremMaximalSFR}). Surprisingly, the techniques that we develop allow us to give a new and simple proof of the fact that generic determinantal rings defined by maximal minors are strongly $F$-regular in positive characteristic (see Theorem \ref{theoremDeterminantalSFR}). The $F$-regularity of determinantal rings was proved by Hochster and Huneke in \cite[\S 7]{HH2} using the following criterion; we briefly review the key ingredients of their proof. 

\begin{thm*}\cite[Fedder-Watanabe's criterion]{FW}
    Let $R$ be an $\mathbb{N}$-graded ring over a field of positive characteristic with homogeneous maximal ideal $\mathfrak{m}$. Then $R$ is $F$-rational if and only if
    \begin{enumerate}
        \item $R$ is Cohen-Macaulay.
        \item The localization $R_{\mathfrak{p}}$ is $F$-rational for all homogeneous prime ideals $\mathfrak{p} \neq \mathfrak{m}$.
        \item The $a$-invariant of $R$ is negative.
        \item $R$ is $F$-injective.
    \end{enumerate}
\end{thm*}

The Cohen-Macaulay property for determinantal rings defined by maximal minors was established using the Eagon-Northcott resolution in \cite{EN} and later in \cite{HE} for minors of all sizes using the technique of principal radical systems. Condition $(2)$ is proved by induction on the size of minors. The $a$-invariant of generic determinantal rings has been computed explicitly (see, for example, \cite{Gr}); alternatively it can be recovered from the main result of \cite{HE}. Condition $(4)$ is argued using the fact that generic determinantal rings are algebras with straightening laws. Alternatively, Sturmfels (\cite{S}) and Caniglia et. al. (\cite{CGG}) independently prove condition $(4)$ by arguing that the minors form a Gr\"obner basis of the determinantal ideal using techniques from algebraic combinatorics. 

Our proof of the $F$-regularity of determinantal rings is ``simple" as the only ingredient used beyond elementary results in linkage theory is Glassbrenner's criterion for $F$-regularity (see Theorem \ref{theoremGlassbrenner}, \S 2). 

The paper is organized as follows: In \S \ref{Sec:Prelims}, we discuss the background on linkage and singularities in positive characteristic needed for our results. In \S \ref{Sec:MainLemma}, we discuss the key tools used to prove our main results. In \S \ref{Sec:MaxMinors}, we prove the $F$-regularity of generic determinantal rings defined by maximal minors. \S \ref{Sec:ResIntSingularity} is dedicated to proving Theorem \ref{theoremGenCU}; we conclude by proving Theorem \ref{theoremGenericSFR}, our main result, in \S \ref{Sec:LinkageSingularity}.

\section{Preliminaries}\label{Sec:Prelims}

\subsection{Linkage and Residual Intersections} 

In this subsection, ``ideal'' will always mean a proper ideal. The symbol $[(a_{i,j})]^T$ is used to denote the transpose of the matrix $[(a_{i,j})]$. 

\begin{defn}\label{definitionLinkage}
    Let $R$ be a Cohen-Macaulay ring, and let $I$ and $J$ be ideals of $R$. We say that $I$ and $J$ are \emph{linked} (or $R/I$ and $R/J$ are linked) if there exists an ideal $\mathfrak{a}$ of $R$ generated by a regular sequence such that \[J = \mathfrak{a}:I \qquad \text{and} \qquad I = \mathfrak{a}:J\] and use the notation $I \sim _{\mathfrak{a}} J$. Furthermore we say that the link is \emph{geometric} if we have $\height(I+J) \geq \height(I)+1$, where $\height(-)$ denotes the height of an ideal of $R$. 
\end{defn}

It is clear that the ideal $\mathfrak{a}$ is contained in $I$ and $J$. Note that the associated primes of $I$ and $J$ have the same height, that is, the ideals $I$ and $J$ are equidimensional. Further note that the heights of the ideals $I$, $J$, and $\mathfrak{a}$ are equal. Moreover, when the link is geometric, it follows that the ideal $\mathfrak{a}$ is the intersection of $I$ and $J$.

Foundational to the study of linkage is the following result of Peskine and Szpiro: 

\begin{prop}\label{propPS}\cite{PS}
    Let $R$ be a Gorenstein ring, $I$ and $J$ be ideals of $R$, and $\mathfrak{a}$ is an ideal generated by a regular sequence such that $I = \mathfrak{a}:J$ and $J = \mathfrak{a}:I$. Suppose that $R/I$ is a Cohen-Macaulay ring. Then 
    \begin{enumerate}
        \item $R/J$ is a Cohen-Macaulay ring.
        \item If $R$ is a local ring, $\omega_{R/I} \cong J/\mathfrak{a}$ and $\omega_{R/J} \cong I/\mathfrak{a}$ where $\omega_{R/I}$ (respectively $\omega_{R/J}$) denotes the canonical module of the ring $R/I$ (respectively $R/J$).
        \item If the ideals $I$ and $J$ are geometrically linked, then $\height(I+J) = \height(I)+1$, and the ring $R/(I+J)$ is Gorenstein. 
    \end{enumerate}
\end{prop}

The following theorem of Peskine and Szpiro says that one can always create a link of an equidimensional ideal in a Gorenstein ring; we present the proof for the convinience of the reader.

\begin{prop}\label{propUnmixed}\cite{PS}
    Suppose that $R$ is a Gorenstein ring and that $I$ is an equidimensional ideal of $R$ of height $g$. Let $\mathfrak{a}$ be an ideal of $R$ generated by a length $g$ regular sequence which is properly contained in the ideal $I$, and let $J = \mathfrak{a}:I$, then $I \sim _{\mathfrak{a}} J$. 
\end{prop}

\begin{proof}
    We only need to show $I = \mathfrak{a}:J$. We begin by observing that we may factor out $\mathfrak{a}$ to assume that $g=0$. This preserves all assumptions due to properties of colon ideals and the fact that $R/\mathfrak{a}$ remains Gorenstein. Thus, we must show that $I = 0:J$. Since $IJ = 0$, it is clear that $I \subseteq 0:J$. Thus, in order to show equality, it is enough to show the equality locally at every associated prime of $I$. Let $\mathfrak{p} \in \Ass(R/I)$. As $I$ is equidimensional of height $0$, it follows that $\mathfrak{p}$ is a minimal prime of $R$. Now localizing at $\mathfrak{p}$ and replacing $R$ by $R_{\mathfrak{p}}$, we may assume that $R$ is a local, Artinian, and Gorenstein ring. By dualizing properties of Gorenstein rings, we have the isomorphism \[R/I \cong \Hom_R(\Hom_R(R/I,R),R).\] Since $\Hom_R(R/I,R) = 0:I = J$, we have  
    \[I = \ann_R(R/I) = \ann_R(\Hom_R(J,R)) \supseteq \ann_R(J) = 0:J \qedhere \]
    
    \end{proof}

The specific kinds of links we study in this paper are called \emph{generic links}.

\begin{defn}
    Let $R$ be a Gorenstein ring and $I = (f_1,\dots,f_n)$ an equidimensional ideal of height $g$ in $R$. Let $X$ be an $g \times n$ matrix of indeterminates, and let $\mathfrak{a}$ be the ideal generated by the entries of $X[f_1\dots f_n]^T$. Then, we call the ideal $\mathfrak{a}:IR[X]$ a \emph{generic link} of $I$. 
\end{defn}

Note that a generic link of $I$ is in fact a link (of $I$) as in Definition \ref{definitionLinkage}. While the definition of generic link requires fixing a generating set of $I$, algebraic properties of generic links are independent of the generating set:

\begin{defn}
    Let $(R,I)$ and $(S,J)$ be pairs where $R$ and $S$ are Noetherian rings and $I\subseteq R$, $J \subseteq S$ are ideals. We say $(R,I)$ and $(S,J)$ are \emph{equivalent}, and write $(R,I) \equiv (S,J)$, if there exist finite sets of variables, $X$ over $R$ and $Z$ over $S$, and an isomorphism $\varphi:R[X] \rightarrow S[Z]$ such that $\varphi(IR[X]) = JR[Z]$. 
\end{defn}

\begin{lem}\cite[Proposition 2.4]{HU1}\label{lemLinksEquiv}
    Let $I$ be an equidimensional ideal in a Gorenstein ring $R$. If $J_1 \subseteq R[X_1]$ and $J_2 \subseteq R[X_2]$ are generic links of $I$, then $(R[X_1], J_1) \equiv (R[X_2], J_2)$. 
\end{lem}

Due to the above lemma, we freely use the phrase ``the generic link'' of an ideal in this paper. The generic link of an equidimensional ideal is prime under mild conditions:

\begin{prop}\label{propLinksPrime}\cite[Proposition 2.6]{HU1}
    Let $R$ be a Cohen-Macaulay domain and $I$ an equidimensional ideal of $R$ such that $R/I$ is generically a complete intersection (that is, $(R/I)_\mathfrak{p}$ is a complete intersection for all $\mathfrak{p} \in \Ass(R/I)$). Then the generic link of $I$ is prime. 
\end{prop}

In \S \ref{Sec:ResIntSingularity}, we discuss the singularities of generic residual intersections of ideals generated by regular sequences. Residual intersections generalize linkage, and are defined as follows:

\begin{defn}\label{defResInt}
    Let $R$ be a Cohen-Macaulay ring, and let $I$ be an ideal of $R$. Given an ideal $\mathfrak{a} \subsetneq I$ generated by $s$ elements, we say that $J = \mathfrak{a}:I$ is an \emph{$s$-residual intersection} of $I$ (or $R/J$ is an $s$-residual intersection of $R/I$) if $\height(J) \geq s$. 
\end{defn}

Note that by Proposition \ref{propUnmixed}, the above definition recovers linkage when $R$ is Gorenstein and $I$ is an equidimensional ideal of height $s$. 

As in the case of linkage, we define a generic residual intersection of an ideal as follows:

\begin{defn}
    Let $R$ be a Gorenstein ring and $I = (f_1,\dots,f_n)$ be an ideal of positive height. Given an integer $s \leq n$ such that $\mu(I_\mathfrak{p}) \leq \dim R_\mathfrak{p}$ (where $\mu(-)$ denotes the minimal number of generators) for all $\mathfrak{p} \in V(I)$ with $\dim(R_\mathfrak{p}) \leq s$, let $X$ be an $s \times n$ matrix of indeterminates and let $\mathfrak{a}$ be the ideal generated by the entries of $X[f_1\dots f_n]^T$. Then, we call the ideal $\mathfrak{a}:IR[X]$ a \emph{generic $s$-residual intersection} of $I$. 
\end{defn}

By \cite[Lemma 3.2]{HU4}, we have that the generic residual intersections of $I$ are residual intersections of $I$, as in Definition \ref{defResInt}. Furthermore, as in the case of links, we need not worry about the generating set that we select for $I$ and therefore freely use the phrase ``the generic residual intersection" of an ideal. 

\begin{lem}\cite[Lemma 2.2]{HU5}\label{lemResIntEquiv}
    Let $R$ be a Gorenstein ring and $I = (f_1,\dots,f_n)$ be an ideal of positive height. Given an integer $s \geq n$ such that $\mu(I_\mathfrak{p}) \leq \dim R_\mathfrak{p}$ for all $\mathfrak{p} \in V(I)$ with $\dim(R_p) \leq s$, if $J_1 \subseteq R[X_1]$ and $J_2 \subseteq R[X_2]$ are generic $s$-residual intersections of $I$, then $(R[X_1], J_1) \equiv (R[X_2], J_2)$. 
\end{lem}

We close this subsection with the following elementary but helpful fact:

\begin{lem}\label{lemmaHeight}
    Let $R$ be an equidimensional and catenary ring (for example, a local Cohen-Macaulay ring) and let $b_1,\dots,b_i$ be part of a system of parameters of $R$. Given the natural ring map $\varphi : R \rightarrow R/(b_1,\dots,b_i)$, we have $\height(I) \geq \height(\varphi(I))$ for any ideal $I$ of $R$. 
\end{lem}

\begin{proof}
    We first show that $\height(I+(b_1,\dots,b_i)) \leq \height(I)+i$. 
    Without loss of generality, we may assume that $i=1$. If $b_1$ is contained in a minimal prime of $I$, we are done. So, assume that $b_1$ is not contained in any minimal prime of $I$. Let $\mathfrak{p}$ be a minimal prime of $I$. Since $R$ is catenary and $b_1$ is not in $\mathfrak{p}$, $\height(\mathfrak{p}+(b_1)) =\height(\mathfrak{p})+1$. As $R$ is equidimensional, this gives us that $\height(I+(b_1)) \leq \height(I)+1$.
    
    Let $\mathfrak{b} = (b_1,\dots,b_i)$. As $R$ is catenary and $b_1,\dots,b_i$ are part of a system of parameters, we have \begin{eqnarray*}
        \height(I) &\geq& \height(I+\mathfrak{b})-i\\
        &=& \dim(R) - \dim\big(R/(I+\mathfrak{b})\big)-i\\
        &=& \dim(R) - \big(\dim(R/\mathfrak{b}) - \height\big((I+\mathfrak{b})/\mathfrak{b}\big)\big)-i \\
        &=& \height\big((I+\mathfrak{b})/\mathfrak{b}\big) + i-i\\
        &=& \height(\varphi(I)), 
    \end{eqnarray*} 
    as claimed.
\end{proof}

\subsection{Singularities in positive characteristic.} 

Let $R$ be a Noetherian ring of prime characteristic $p > 0$. The letter $e$ denotes a variable nonnegative integer, and $q$ is the $e$-th power of $p$, i.e., $q = p^e$. For an ideal $I =(x_1, \ldots, x_n)$ of $R$, let $I^{[q]} =(x_1^q, \ldots, x_n^q)$.

For a reduced ring $R$ of characteristic $p > 0$, $R^{1/q}$ shall denote the ring obtained by adjoining all $q$th roots of elements of $R$. A ring $R$ is said to be \emph{$F$-finite} if $R^{1/p}$ is module-finite over $R$. A finitely generated algebra $R$ over a field $K$ is $F$-finite if and only if $K^{1/p}$ is a finite field extension of $K$ --- a fairly mild condition.

\begin{defn}
Let $R$ be a ring of characteristic $p>0$, $R^{\circ}$ the complement of the union of its minimal primes, and $I$ an ideal of $R$. For an element $x$ of $R$, we say that $x \in I^F$, the Frobenius closure of $I$, if there exists $q = p^e$ such that $x^q \in I^{[q]}$.

An element $x$ of $R$ is in $I^*$, the tight closure of $I$, if there exists $c\in R^{\circ}$ such that $cx^q \in I^{[q]}$ for all $q = p^e \gg 0$. If $I = I^*$, the ideal $I$ is tightly closed. Note that $I \subseteq I^F \subseteq I^*$.

A ring $R$ is \emph{$F$-pure} if the Frobenius map is pure. That is, $F: M \longrightarrow F(M)$ is injective for all $R$-modules $M$. Note that this implies $I^F = I$ for all ideals $I$ of $R$. 

A graded or local ring $(R,\mathfrak{m})$ is \emph{$F$-injective} if the natural Frobenius action on the local cohomology module $F: H^i_{\mathfrak{m}}(R) \longrightarrow H^i_{\mathfrak{m}}(R)$ is injective for each integer $i$.

A ring $R$ is weakly $F$-regular if every ideal of $R$ is tightly closed, and is \emph{$F$-regular} if every localization is weakly $F$-regular. An $F$-finite ring $R$ is strongly $F$-regular if for every element $c \in R^{\circ}$, there exists an integer $q = p^e$ such that the $R$-linear inclusion $R \longrightarrow R^{1/q}$ sending $1$ to $c^{1/q}$ splits as a map of $R$-modules. $R$ is $F$-rational if, in every local ring of $R$, all ideals generated by systems of parameters are tightly closed.
\end{defn}

It is clear that a graded (or local) $F$-pure ring is $F$-injective. We summarize some basic results regarding these notions from \cite[Theorem 2.1]{K} \cite[Theorem 3.1]{HH1}, \cite[Theorem 4.3]{HH3}, \cite[Lemma 3.3]{F}, and \cite[Corollaries 4.3 \text{and} 4.4]{LS} which we use. 

\begin{thm} \label{TheoremF-singularities}\ 
    \begin{enumerate} 
        \item A ring $R$ is regular if and only if the Frobenius map on $R$ is flat.
    
        \item Regular rings are $F$-regular; if they are $F$-finite, they are also strongly $F$-regular.

        \item An $F$-rational Gorenstein ring is $F$-regular. If it is $F$-finite, then it is also strongly $F$-regular.
        
        \item An $F$-injective Gorenstein ring is $F$-pure.
        
        \item The notions of weak $F$-regularity and $F$-regularity agree for $\mathbb{N}$-graded rings. For $F$-finite $\mathbb{N}$-graded rings, these are also equivalent to strong $F$-regularity.
    \end{enumerate}
\end{thm}

The following criterion of Fedder is useful in determining the $F$-purity of finitely generated $K$-algebras:

\begin{thm}\label{theoremFedder}\cite[Theorem 1.12]{F}
Let $R = K[x_1, \ldots, x_n]$ be an $\mathbb{N}$-graded polynomial ring over a field $K$ of positive characteristic, $I$ is a homogeneous ideal of $R$ and $\mathfrak{m} = (x_1,\ldots, x_n)$. Then, $R/I$ is $F$-pure if and only if for some positive integer $e$, $(I^{[p^e]}:I) \nsubseteq \mathfrak{m}^{[p^e]}$.
\end{thm}

As an immediate application of the above theorem, one can establish precisely when complete intersection rings are $F$-pure.

\begin{prop}\label{propFedderCI} \cite[Proposition 2.1]{F}
    Let $(R, \mathfrak{m})$ be a graded (or local) regular ring of prime characteristic $p$ and $\mathfrak{a}$ be an ideal of $R$ generated by a regular sequence $f_1, \ldots, f_g$, and $f = \prod _{i=1}^g f_i$. The following are equivalent:
    \begin{enumerate}
        \item $R/\mathfrak{a}$ is an $F$-pure ring.
        \item $R/f$ is an $F$-pure ring.
        \item $f^{p-1} \notin \mathfrak{m}^{[p]}$.
    \end{enumerate}
\end{prop}

\begin{proof}
    Note that $\mathfrak{a}^{[p]}:\mathfrak{a} = f^{p-1} + \mathfrak{a}^{[p]}$. The proof is immediate from Theorem \ref{theoremFedder}.
\end{proof}

The following criterion of Glassbrenner is useful in determining the $F$-regularity of finitely generated $K$-algebras: 

\begin{thm}\label{theoremGlassbrenner} \cite[Theorem 3.1]{G}
    Let $R$ be an $\mathbb{N}$-graded domain of prime characteristic $p$ such that $[R]_0$ is an infinite $F$-finite field and $R$ is a finitely generated $[R]_0$ algebra. Write \[R \cong K[x_1 \ldots , x_n]/I\]
    where $K =[R]_0$, $K[x_1, \ldots, x_n]$ is an $\mathbb{N}$-graded polynomial ring over $K$ and $x_1, \ldots x_n$ are homogeneous, and $I$ is a homogeneous ideal.
    
    Let $s$ be a nonzero, homogeneous element of $K[x_1 \ldots, x_n]$ not in $I$ for which $R_s$ is regular or even just strongly F-regular. Let $\mathfrak{m}$ be the homogeneous maximal ideal of $K[x_1 \ldots, x_n]$. The following are equivalent:
    \begin{enumerate}
        \item The ring $R$ is strongly $F$-regular.
        \item There exists a positive integer $e$ such that $s(I^{[p^e]}:I) \nsubseteq \mathfrak{m}^{[p^e]}$.
    \end{enumerate}
\end{thm}

\subsection*{Convention and notation:}
Throughout the paper, we assume that all rings of positive prime characteristic under consideration are $F$-finite. All graded rings are assumed to have a unique maximal ideal. Since the rings we work with are $\mathbb{N}$-graded, we make no distinction between the various notions of $F$-regularity in view of Theorem \ref{TheoremF-singularities} ($5$). 

Let $X= (x_{i,j})$ be a $t\times n$ matrix of indeterminates (unless stated otherwise) over a field $K$ with $t \leq n$. Let $I_t(X)$ denote the ideal generated by the $t \times t$ minors of $X$; the quotient ring $K[X]/I_t(X)$ is called the generic determinantal ring defined by minors of size $t$. It is well-known that the ideal of maximal minors $I_t(X)$ has height $n-t+1$ in $K[X]$. We use $[i,t+i-1]$ to denote the size $t$ minor of $X$ with columns $i$ through $t+i-1$. The symbol $[(a_{i,j})]^T$ is used to denote the transpose of the matrix $[(a_{i,j})]$.

\section{Fedder's Criterion and Linkage}\label{Sec:MainLemma}

\begin{lem}\label{lemmaMain}
    Let $(R,\mathfrak{m})$ be a regular graded (or local) ring of prime characteristic $p$. Let $\mathfrak{a}$ and $I$ be proper ideals of $R$. If $J = \mathfrak{a}:I$ is a proper ideal of $R$, then the ideal $\mathfrak{a}^{[p]}:\mathfrak{a}$ is contained in $J^{[p]}:J$. In particular, if $R/\mathfrak{a}$ is $F$-pure, then so is $R/J$.
\end{lem}

\begin{proof}
     We have the following chain of containments of ideals: 
\begin{eqnarray*}
        \mathfrak{a}^{[p]}:\mathfrak{a} &\subseteq& \mathfrak{a}^{[p]}: IJ \\
        &=&  (\mathfrak{a}^{[p]}: I): J \\
        &\subseteq&  (\mathfrak{a}^{[p]}: I^{[p]}): J \\
        &=&  (\mathfrak{a}: I)^{[p]}: J \\
        &=&  J^{[p]}: J,
\end{eqnarray*}

    where the second to last equality follows from the flatness of the Frobenius map on $R$ (see Theorem \ref{TheoremF-singularities} ($1$)). By Fedder's criterion (Theorem \ref{theoremFedder}), it follows that if $R/ \mathfrak{a}$ is $F$-pure, then so is $R/J$.
\end{proof}

When ideals $I$ and $J$ are geometrically linked by a regular sequence $\mathfrak{a}$, the $F$-singularity of $\mathfrak{a}$ controls the $F$-singularities of $I$ and $J$.   

\begin{cor}
    Let $(R,\mathfrak{m})$ be a graded (or local) regular ring of prime characteristic $p$. Let $I$ and $J$ be proper ideals of $R$. Suppose that $R/I$ is a Cohen-Macaulay ring and that $I$ and $J$ are geometrically linked. If the rings $R/I$, $R/J$, and $R/(I+J)$ are $F$-injective, then they are each $F$-pure. 
\end{cor}

\begin{proof}
    By Proposition \ref{propPS}, the ring $R/J$ is Cohen-Macaulay. As $I$ and $J$ are geometrically linked, there exists an ideal $\mathfrak{a}$ generated by an $R$-regular sequence such that $J = \mathfrak{a}:I$ and $I = \mathfrak{a}:J$. Thus, we have the following short exact sequence:
    \[0 \rightarrow R/\mathfrak{a} \rightarrow R/I \oplus R/J \rightarrow R/(I+J) \rightarrow 0.\] 
    
     Let $d$ denote the Krull dimension of $R/I$. By Proposition \ref{propPS}, we get $\dim(R/J)=d$ and $\dim(R/(I+J)) = d-1$. The above short exact sequence induces an exact sequence of local cohomology modules with support in $\mathfrak{m}$. The natural Frobenius action on these local cohomology modules gives the following commutative diagram: 

    {\centering
    \begin{tikzcd}
0\arrow[r] & H_\mathfrak{m}^{d-1}(R/(I+J)) \arrow[r] \arrow[d, "F^e"] &  H_\mathfrak{m}^{d}(R/\mathfrak{a}) \arrow[r] \arrow[d,"F^e"] & H_\mathfrak{m}^{d}(R/I) \oplus H_\mathfrak{m}^{d}(R/J) \arrow[r]\arrow[d,"F^e"]  & 0\\
0\arrow[r] & H_\mathfrak{m}^{d-1}(R/(I+J)) \arrow[r] &  H_\mathfrak{m}^{d}(R/\mathfrak{a}) \arrow[r] & H_\mathfrak{m}^{d}(R/I) \oplus H_\mathfrak{m}^{d}(R/J) \arrow[r]  & 0.
\end{tikzcd}

}
As the rings $R/I$, $R/J$, and $R/(I+J)$ are $F$-injective, by the short five lemma we have that the natural Frobenius action on $H_\mathfrak{m}^{d}(R/\mathfrak{a})$ is injective, and thus $R/\mathfrak{a}$ is $F$-injective. Since $\mathfrak{a}$ is a complete intersection, by Theorem \ref{TheoremF-singularities} ($4$), it is also $F$-pure. Thus $R/I$ and $R/J$ are $F$-pure by Lemma \ref{lemmaMain}.

Further, by Lemma \ref{lemmaMain}, we have
\begin{eqnarray*}
    (I+J)^{[p]}:(I+J) &=& ((I+J)^{[p]}:I) \cap ((I+J)^{[p]}:J) \\
    &\supseteq& (J^{[p]}:J) \cap (I^{[p]}:I) \\
    &\supseteq& \mathfrak{a}^{[p]}:a .
\end{eqnarray*}

It follows from Fedder's criterion (Theorem \ref{theoremFedder}) that the ring $R/(I+J)$ is $F$-pure.
\end{proof}

\begin{cor}\label{corMain}
    Let $(R,\mathfrak{m})$ be a graded (or local) regular ring of prime characteristic $p$ and let be $I$ an equidimensional ideal of $R$. If $\mathfrak{a} \subsetneq I$ is an ideal generated by a regular sequence of length $\height (I)$, then the ideal $\mathfrak{a}^{[p]}:\mathfrak{a}$ is contained in $I^{[p]}:I$. In particular, if $R/\mathfrak{a}$ is $F$-injective, then $R/I$ is $F$-pure. 
\end{cor}

\begin{proof}
    By Proposition \ref{propUnmixed}, we have $I = \mathfrak{a}:(\mathfrak{a}:I)$. Thus, we are done by Lemma \ref{lemmaMain} in view of Theorem \ref{TheoremF-singularities} (4).
\end{proof}

The following statement is true more generally for any $F$-singularity; it is useful for us when dealing with the $F$-regularity of generic links and residual intersections.

\begin{lem}\label{lemmaAnyGensOkay}
    Let $(R,\mathfrak{m})$ and $(S,\mathfrak{n})$ be graded (or local) regular rings. Let $I\subseteq R$ and $J \subseteq S$ be ideals. If $(R,I) \equiv (S,J)$, then $R/I$ is (strongly) $F$-regular if and only if $S/J$ is (strongly) $F$-regular. 
\end{lem}

\begin{proof}
The assertion follows from the fact that a ring $A$ is (strongly) $F$-regular if and only if $A[x]$ is (strongly) $F$-regular where $x$ is an indeterminate; see \cite[Chapter 7]{MP}.
\end{proof}

\section{\texorpdfstring{$F$}{F}-Regularity of Determinantal Rings\\ Defined by Maximal Minors}\label{Sec:MaxMinors}

While it has been shown that generic determinantal rings are $F$-regular in positive characteristic in \cite[\S 7]{HH2}, the proof is fairly difficult and requires the application of several deep results as discussed in the introduction. We provide a simple proof for the $F$-regularity of generic determinantal rings defined by maximal minors by applying Corollary \ref{corMain} along with Glassbrenner's Criterion (Theorem \ref{theoremGlassbrenner}).

In order to apply Corollary \ref{corMain}, we begin with finding a regular sequence in the determinantal ideal $I_t(X)$ of the appropriate length.

\begin{lem}\label{lemmaRegularSecDet}
    Let $X = (x_{i,j})$ be a $t\times n$ matrix of indeterminates where $n \geq t$, and let $K[X]$ be a polynomial ring over a field $K$. Then \[[1,t],[2,t+1],\dots,[n-t+1,n]\] is a regular sequence of length $n-t+1$ contained in the ideal $I_t(X)$. 
\end{lem}

\begin{proof}
    Define a map $\varphi: K[X] \rightarrow k[x_{1,t},x_{1,t+1},\dots,x_{1,n}]$ such that \[\varphi(x_{i,j}) = \begin{cases}
        x_{1,i+j-1} & \text{for } t \leq i+j -1 \leq n\\
        0 & \text{otherwise.}
    \end{cases}\]

    Notice that kernel of $\varphi$ is generated by \[\{x_{i,j} -  x_{1,i+j-1} \; \vert \; t \leq i+j -1 \leq n\} \cup \{ x_{i,j} \; \vert \;  \text{otherwise }\}, \] which is clearly a regular sequence in $K[X]$. Furthermore, notice that the image of  the matrix $X$ in $K[x_{1,t},x_{1,t+1},\dots,x_{1,n}]$ is

    \[\begin{bNiceMatrix}
        0 & \Ldots  &  & 0& x_{1,t} & x_{1,t+1} & \Ldots & x_{1,n}\\
        0 & \Ldots & 0 & x_{1,t}& x_{1,t+1}  & \Ldots & x_{1,n} & 0 \\
        \Vdots & \Iddots & \Iddots & \Iddots  &  & \Iddots & \Iddots & \Vdots \\
        0 & x_{1,t} & x_{1,t+1}& \Ldots & x_{1,n} & 0 & \Ldots & 0\\
        x_{1,t} & x_{1,t+1}& \Ldots & x_{1,n} & 0 & \Ldots & & 0
    \end{bNiceMatrix}.\]

    On expanding determinants along the first column, we get \[\begin{array}{rcl}
         \varphi([1,t]) &=& x_{1,t}^t, \\
         \varphi([2,t+1]) &=&  x_{1,t+1}^t + f_2,\\
         &\vdots& \\
         \varphi([i,t+i-1]) &=& x_{1,t+i-1}^t + f_{i},\\
         &\vdots& \\
         \varphi([n-t+1,n]) &=& x_{1,n}^t + f_n,
    \end{array}\]
    where $f_i \in (x_{1,t}, \dots, x_{1,t+i-2})$. Let $\underline{\alpha} = \{[1,t],[2,t+1],\dots,[n-t+1,n]$\}; it follows that $\sqrt{(\varphi(\underline{\alpha}))} = (x_{1,t},\dots,x_{1,n})$. By Lemma \ref{lemmaHeight}, we have that \[n-t+1 \geq\height((\underline{\alpha})) \geq \height(\varphi(\underline{\alpha})) = n-t+1,\] so we have equality throughout.
\end{proof}

The following lemma will help us apply Glassbrenner's Criterion (Theorem \ref{theoremGlassbrenner}). 

\begin{lem}\label{lemmaNotInDet}
    Let $X = (x_{i,j})$ be a $t\times n$ matrix of indeterminates where $n \geq t>1$, let $K[X]$ be a polynomial ring such that $K$ is a field of characteristic $p>0$, and let $\mathfrak{m}$ be the homogeneous maximal ideal. Then \[x_{1,n}([1,t][2,t+1]\dots[n-t+1,n])^{p-1} \notin m^{[p]}.\] 
\end{lem}

\begin{proof}
    First notice that $\mathfrak{m}^{[p]}$ is a monomial ideal. Thus, for $f \in K[X]$, $f \in \mathfrak{m}^{[p]}$ if and only if each monomial term of $f$ is in $m^{[p]}$.
    
    We will use the lexicographical order induced from the variable order \[x_{1,1} > x_{1,2} > \dots > x_{1,n} > x_{2,1} > \dots > x_{2,n} > \dots > x_{t,n}.\] For a polynomial $f$, let $\initial_<(f)$ denote the initial monomial of $f$ with respect to this order. By determinant expansion along the first row, we get \[\initial_<\big([i,t+i-1]\big) = x_{1,i}x_{2,i+1}\dots x_{t,t+i-1},\] which is just the product along the main diagonal of the minor $[i,t+i-1]$. Recall that for polynomials $f$ and $g$, $\initial_<(fg) =\initial_<(f)\initial_<(g)$. Thus 
    \begin{eqnarray*}
        \initial_<\big(x_{1,n}([1,t][2,t+1]\dots[n-t+1,n])^{p-1}\big) &=& x_{1,n}\big(\initial_<([1,t])\dots\initial_<([n-t+1,n])\big)^{p-1}\\
        &=& x_{1,n}\Bigg(\prod_{0 \leq j-i \leq n-t} x_{i,j}^{p-1}\Bigg).
    \end{eqnarray*}
We get $\initial_<\bigg(x_{1,n}\big([1,t][2,t+1]\dots[n-t+1,n]\big)^{p-1}\bigg) \notin m^{[p]}$, and we are done.
\end{proof}

The last ingredient that we need to prove the main result of this section is the following well-known lemma (see, for example, \cite[Proposition 2.4]{BV}). We include the proof since it is an elementary idea used repeatedly in this paper.

\begin{lem}\label{lemmaInvertx11}
    Consider the matrices $X = (x_{i,j})$ of indeterminates where $1 \leq i \leq m$, $1 \leq j \leq n$, and $Y = (y_{i,j})$ where $2 \leq i \leq m$, $2 \leq j \leq n$. Set $R = K[X]$ and $R' = K[Y]$. Then the map
    
    \[R'[x_{1,1,}, \ldots, x_{m,1},x_{1,2}, \ldots, x_{1,n}]_{x_{1,1}} \longrightarrow R_{x_{1,1}} \quad \text{with} \quad y_{i,j} \mapsto x_{i,j} - x_{i,1}x_{1,j}x_{1,1}^{-1} \]
    is an isomorphism. Moreover, $R_{x_{1,1}}$ is a free $R'$-module, and for each $t \geq 1$, one has
    \[I_t(X)R_{x_{1,1}} = I_{t-1}(Y)R_{x_{1,1}}\]
    under this isomorphism.
\end{lem}

\begin{proof}
    We may invert $x_{1,1}$ and perform elementary row operations to transform $X$ into a matrix where $x_{1,1}$ is the only nonzero entry in the first column. After this, one may perform elementary column operations to obtain a matrix 
    \[\begin{bmatrix}
        x_{1,1} & 0 & \hdots & 0 \\
        0 & x'_{2,2} & \hdots & x'_{2,n}\\
        \vdots & \vdots & & \vdots \\
        0 & x'_{m,2} & \hdots & x'_{m,n}
    \end{bmatrix} \quad \text{where} \quad x'_{i,j} = x_{i,j} - x_{i,1}x_{1,j}x_{1,1}^{-1}; \]
    the asserted isomorphism is then $y_{i,j} \mapsto x'_{i,j}$. Note that the ideal $I_t(X)R_{x_{1,1}}$ is generated by the size $t$ minors of the displayed matrix, and hence equals $I_{t-1}(Y)R_{x_{1,1}}$. Finally, $R_{x_{1,1}}$ is a free $R'$-module since the ring extension
    \[K[x_{i,j} - x_{i,1}x_{1,j}x_{1,1}^{-1} | 2 \leq i \leq m, 2 \leq j \leq n] \subset K[X,x_{1,1}^{-1}]\]
    is obtained by adjoining indeterminates $x_{1,1,}, \ldots, x_{m,1},x_{1,2}, \ldots, x_{1,n}$ and inverting $x_{1,1}$.
\end{proof}

\begin{thm}\label{theoremDeterminantalSFR}
    Let $X = (x_{i,j})$ be a $t\times n$ matrix of indeterminates where $n \geq t$, and let $K[X]$ be a polynomial ring such that $K$ is (an $F$-finite field) of characteristic $p>0$, then the generic determinantal ring $R = K[X]/I_{t}(X)$ is (strongly) $F$-regular. 
\end{thm}

\begin{proof}
    We proceed by induction on $t$. The statement is clear for $t = 1$. Now, we suppose that the statement is true for $t-1 \geq 1$ and show that it is true for $t$. To do this, we apply Glassbrenner's criterion (Theorem \ref{theoremGlassbrenner}) with $s = x_{1,n}$. Since $I_t(X)$ is prime \cite[Theorem 2.10]{BV}, it is clear that $x_{1,n} \notin I_t(X)$. Note that, by Lemma \ref{lemmaInvertx11}, we have the isomorphism \[R_{x_{1,n}} \cong \big( K[Y]/I_{t-1}(Y)\big)[x_{1,1},\dots,x_{1,n},x_{2,n},\dots,x_{t,n}][x_{1,n}^{-1}],\] where $Y = (y_{i,j})$ is a $(t-1) \times (n-1)$ matrix of indeterminates. It follows from induction that $R_{x_{1,n}}$ is $F$-regular. 
    
    Next, we must show that \[x_{1,n}(I_t(X)^{[p]}:I_t(X)) \not\subseteq \mathfrak{m}^{[p]}.\] Crucially, by Corollary \ref{corMain}, it is enough to show that $x_{1,n}(\mathfrak{a}^{[p]} : \mathfrak{a}) \not\subseteq \mathfrak{m}^{[p]}$ for some ideal $\mathfrak{a}$ contained in $I_t(X)$ generated by a length $n-t+1$ regular sequence. We use Lemma \ref{lemmaRegularSecDet} to choose \[\mathfrak{a} = ([1,t],[2,t+1],\dots,[n-t+1,n]).\] We are now done by Lemma \ref{lemmaNotInDet}. 
\end{proof}

\begin{rem}\label{RemarkDetFpure}
    Notice that the $F$-purity of generic determinantal rings defined by maximal minors directly follows from Fedder's criterion (Theorem \ref{theoremFedder}) on applying Corollary \ref{corMain} along with Lemmas \ref{lemmaRegularSecDet} and \ref{lemmaNotInDet}. \end{rem}

    It immediately follows from Theorem \ref{theoremDeterminantalSFR} and Lemma \ref{lemmaMain} that the link of the determinantal ring defined by maximal minors, that is, the ring $K[X_{t \times n}]/(\mathfrak{a}:I_t(X))$ is $F$-pure for the choice of regular sequence \[\mathfrak{a} = ([1,t],[2,t+1],\dots,[n-t+1,n]).\] The main result of our paper, Theorem \ref{theoremGenericSFR}, says that this link (or any other link) of the determinantal ring defined by maximal minors has a deformation---the generic link of $K[X_{t \times n}]/I_t(X)$---which is strongly $F$-regular.

\section{\texorpdfstring{$F$}{F}-Regularity of the Generic Residual Intersections \\ of Graded Complete Intersection Rings} \label{Sec:ResIntSingularity}

In this section, we prove that the generic residual intersections of a graded and $F$-rational complete intersection ring is strongly $F$-regular. This strengthen the result \cite[Proposition 3.4]{CU} under the assumption that the regular sequence consists of homogeneous elements. The bulk of the work for this is done in the following theorem where the ideal $I$ is the homogeneous maximal ideal of $R$.   

\begin{thm}\label{theoremMaximalSFR}
     Let $R = K[x_1, \ldots, x_n]$ be a polynomial ring over a field, $\mathfrak{m}$ be the homogeneous maximal ideal of $R$, and $U$ be a matrix of indeterminates over $R$. Then,
     \begin{enumerate}
         \item If $K$ is a field of characterstic zero and $J \subseteq S = R[U]$ is a generic residual intersection of $\mathfrak{m}$, then $S/J$ has rational singularities.
         
         \item If $K$ is an $F$-finite field of positive characteristic and $J \subseteq S = R[U]$ is a generic residual intersection of $\mathfrak{m}$, then $S/J$ is strongly $F$-regular.
     \end{enumerate}
\end{thm}

\begin{rem}\label{remarkMaximalFixGens}
By Lemma \ref{lemResIntEquiv} and Lemma \ref{lemmaAnyGensOkay}, $F$-regularity and all other relevant properties are independent of the choice of a generating set of $\mathfrak{m}$. So, for the rest of this section, we fix the generators of $\mathfrak{m}$ to be $x_1,\dots,x_n$.

Given this, for $s \geq n$, let $U=(u_{i,j})$ be an $s \times n$ matrix of indeterminates, and let $\mathfrak{a} \subseteq \mathfrak{m}R[U]$ be the ideal generated by the entries of $U[x_1 \dots x_n]^T$. We fix $J = \mathfrak{a}:\mathfrak{m}R[U]$ to be our choice of a generic $s$-residual intersection. 
\end{rem}

The proof of the above theorem will consist of applying Corollary \ref{corMain} along with Glassbrenner's Criterion (Theorem \ref{theoremGlassbrenner}). Therefore, we again begin by finding an appropriate regular sequence of length $\height(J)$ contained in in the ideal $J$. 

\begin{lem}\label{lemmaRegularSecMax}
    Let $U=(u_{i,j})$ be an $s\times n$ matrix of indeterminants for $n>1$, and let $K[x_1,\dots,x_n][U]$ be a polynomial ring over a field $K$. Let $\Delta_1,\dots,\Delta_{s-n+1}$ be the set of $n \times n$ minors of $U$ consisting of adjacent rows. If $\underline{\beta}$ is the sequence consisting of the entries of \[\begin{bmatrix}
        u_{1,1} &\hdots& u_{1,n}\\
        \vdots && \vdots \\
        u_{n-1,1} &\hdots& u_{n-1,n}
    \end{bmatrix} \begin{bmatrix}
        x_1\\
        \vdots \\
        x_n
    \end{bmatrix}\] and $\Delta_1,\dots,\Delta_{s-n+1}$, then $\underline{\beta}$ is a regular sequence. 
\end{lem}

\begin{proof}
Let $W$ be an $n\times s$ matrix of indeterminates over $K[x_1,\dots,x_n]$ and let $\underline{\gamma}$ be the sequence of the entries of \[[x_1\dots x_n]\begin{bmatrix}
        w_{1,1} &\hdots& w_{1,n}\\
        \vdots && \vdots \\
        w_{n-1,1} &\hdots& w_{n-1,n}
    \end{bmatrix}\] and $[1,n],\dots,[s-n+1,s]$, where $[i,i+n-1]$ are the $n\times n$ minors of $W$ consisting of columns $i$ through $i+n-1$. Notice that showing $\underline{\beta}$ is a regular sequence in $K[x_1,\dots,x_n][U]$ is equivalent to showing $\underline{\gamma}$ is a regular sequence in $K[x_1,\dots,x_n][W]$. So, for notational symmetry with Lemma \ref{lemmaRegularSecDet}, we will show that $\underline{\gamma}$ is a regular sequence.

Define a map $\varphi: K[x_1,\dots,x_n][W] \rightarrow K[W]$ such that \[\varphi(w_{i,j}) = \begin{cases}
    w_{1,i+j-1} & \text{for } n \leq i+j-1 \leq s \\
    w_{i,j} & \text{for } j  = n-i\\
    0 & \text{otherwise,}
\end{cases}\] and 
\[\varphi(x_i) = \begin{cases}
    0 & \text{for } i = n \\
    w_{i,n-i} & \text{otherwise.}
\end{cases}\]

The kernel of $\varphi$ is generated by the union of the sets \[\{w_{i,j}-w_{1,i+j-1} \;\vert \; n \leq i+j-1 \leq s\}, \{w_{i,j} \; \vert \; i+j < n \text{ or } i+j > s+1\},\{x_i-w_{i,n-i} \; \vert \; i \neq n\}, \text{and} \{x_n\},\] which is clearly a regular sequence. Furthermore, notice that \[\varphi([x_1 \dots x_n]) = [w_{1,n-1} \; w_{2,n-2}\dots w_{n-1,1}\;0]\] and \[\varphi(W)  = \begin{bNiceMatrix}
        0 & \Ldots && 0 & w_{1,n-1}& w_{1,n} & w_{1,n+1} & \Ldots & w_{1,s}\\
        0 & \Ldots &0& w_{2,n-2} & w_{1,n}& w_{1,n+1}  & \Ldots & w_{1,s} & 0 \\
        \Vdots &\Iddots& \Iddots & \Iddots & \Iddots & & \Iddots & \Iddots & \Vdots \\
        0&  &  & &  &  & &  & \\
        w_{n-1,1} & w_{1,n} & w_{1,n+1}& \Ldots & w_{1,s} & 0 & & \Ldots & 0\\
        w_{1,n} & w_{1,n+1}& \Ldots & w_{1,s} & 0 & & \Ldots &  & 0
    \end{bNiceMatrix}.\]

Note that each of the diagonals in the above display is constant, except for the first nonzero diagonal (from the left), which reads $w_{1,n-1}, w_{2,n-2}, \ldots, w_{n-1,1}$. This gives us \[\varphi(w_{i,1}x_1 + \dots w_{i,n}x_n ) = w_{n-i,i}^2 + f_i\] for $f_i \in (w_{n-i+1,i-1}, \dots w_{n-1,1})$ and $1 \leq i \leq n-1$ (where $f_1=0$) and, by expanding determinants along the last column,
\[\varphi\big([i,n+i-1]\big) = w_{1,n+i-1}^n + g_{i}\] for $g_i \in (w_{1,n+i},\dots, w_{1,s})$ and $1 \leq i \leq s-n+1$ (where $g_{s-n+1} = 0$). This gives us that $\sqrt{\mathfrak{(\underline{\gamma})}} = (w_{1,n-1},w_{2,n-2},\dots,w_{n-1,1},w_{1,n},w_{1,n+1},\dots,w_{1,s})$. By Lemma \ref{lemmaHeight}, we have that \[s \geq \height\big((\underline{\gamma})\big) \geq \height\big(\varphi\big((\underline{\gamma})\big)\big) = s,\] so we have equality throughout. 
\end{proof}

\begin{rem}\label{remarkGensCI}
    By \cite[Example 3.4]{HU4}, $J = \mathfrak{a}+I_n(U)$ and $\height(J) = s$. So, $\mathfrak{b} = (\beta)$ is an ideal generated by a length $s$ regular sequence properly contained in $J$. 
\end{rem}

The following lemma will help us apply Glassbrenner's Criterion (Theorem \ref{theoremGlassbrenner}). 

\begin{lem}\label{lemmaNotInMax}
    Let $U=(u_{i,j})$ be an $s\times n$ matrix of indeterminates for $n>1$, and let $S = K[x_1,\dots,x_n][U]$ be a polynomial ring over a field $K$. Let $\Delta_1,\dots,\Delta_{s-n+1}$ be the set of $n \times n$ minors of $U$ consisting of adjacent rows. If $\mathfrak{n}$ is the homogeneous maximal ideal of $S$, then \[x_1(\Delta_1\dots \Delta_{s-n+1})^{p-1}\prod_{i=1}^{n-1}(u_{i,1}x_1+\dots+u_{i,n}x_n)^{p-1} \notin \mathfrak{n}^{[p]}.\] 
\end{lem}

\begin{proof}
    Without loss of generality, we let
    \[\Delta_{i} = \det \begin{bmatrix}
        u_{i,1} & \hdots & u_{i,n} \\
        \vdots & & \vdots \\
        u_{i+n-1,1} & \hdots & u_{i+n-1,n}
    \end{bmatrix}.\]

    Recall that $\mathfrak{n}^{[p]}$ is a monomial ideal. Thus, for polynomials $f \in S$, $f \in \mathfrak{n}^{[p]}$ if and only if every monomial term of $f$ is in $n^{[p]}$. So it is enough to show that \[\initial_<\big(x_1(\Delta_1\dots \Delta_{s-n+1})^{p-1}\prod_{i=1}^{n-1}(u_{i,1}x_1+\dots+u_{i,n}x_n)^{p-1}\big) \notin \mathfrak{n}^{[p]}\] for some monomial ordering.
    
    Use the lexicographical order induced from the following variable order:
    
    For $u_{i,j} \neq u_{l,k}$, set
    \begin{eqnarray*}
        u_{i,j} > u_{l,k} & \text{ if }  & \begin{cases}
            i>l \\
            i = l \text{ and } \begin{cases}
                j = i+1 \\
                j>k \text{ and } k \neq i+1.
            \end{cases}
        \end{cases}
    \end{eqnarray*}
    Furthermore, let $u_{i,j} > x_k$ for all $i,j,$ and $k$ and fix an arbitrary order on the $x_i$. For a polynomial $f$, let $\initial_<(f)$ represent the initial monomial of $f$ with respect to this order. By expanding determinants along the bottom row, one can see that \[\initial_<(\Delta_{i}) = u_{i,1}u_{i+1,2}\dots u_{i+n-1,n},\] which is just the product along the main diagonal. Furthermore,
    \[\initial_<(u_{i,1}x_1+\dots+u_{i,n}x_n) = u_{i,i+1}x_{i+1}.\]
    
     Recall that, for arbitrary polynomials $f$ and $g$, $\initial_<(fg) =\initial_<(f)\initial_<(g)$. Thus

\begin{eqnarray*}
    &&\initial_<\big(x_1(\Delta_1\dots \Delta_{s-n+1})^{p-1}\prod_{i=1}^{n-1}(u_{i,1}x_1+\dots+u_{i,n}x_n)^{p-1}\big)
    \\ &=& x_1\big(\initial_<(\Delta_1\dots \Delta_{s-n+1})\big)^{p-1}\prod_{i=1}^{n-1}\big(\initial_<(u_{i,1}x_1+\dots+u_{i,n}x_n)\big)^{p-1}
    \\ &=&x_1\Big(\prod_{0 \leq i-j \leq s-n} u_{i,j}^{p-1}\Big)(x_2\dots x_n u_{1,2}u_{2,3}\dots u_{n-1,n})^{p-1} \\ &=& x_1(x_2\dots x_n)^{p-1}\Big(\prod_{-1 \leq i-j \leq s-n} u_{i,j}^{p-1}\Big).
\end{eqnarray*}

The monomial $x_1(x_2\dots x_n)^{p-1}\Big(\prod_{-1 \leq i-j \leq s-n} u_{i,j}^{p-1}\Big)$ is clearly not contained in $\mathfrak{n}^{[p]}$, so we are done. 
\end{proof}

We are now ready to prove Theorem \ref{theoremMaximalSFR}.

\begin{proof}[Proof of Theorem \ref{theoremMaximalSFR}]

    It suffices to prove assertion $(2)$ of the theorem; since $F$-regular rings are $F$-rational, it then follows that $R/I$ is of $F$-rational type for $K$ of characteristic zero, and then by \cite[Theorem 4.3]{Karen} that $R/I$ has rational singularities.

    By Lemma \ref{lemResIntEquiv} and Lemma \ref{lemmaAnyGensOkay}, we may fix $S$, $\mathfrak{a}$, and $J$ to be as in Remark \ref{remarkMaximalFixGens}. Let $\mathfrak{n}$ be the homogeneous maximal ideal of $S$.

    In the case where $n=1$, the ring $S/J = K[x_1][u_{1,1},\dots,u_{s,1}]/(u_{1,1},\dots,u_{s,1}) \cong K[x_1]$ is regular, and so we may assume that $n > 1$. 

    We apply Glassbrenner's criterion (Theorem \ref{theoremGlassbrenner}) with $s=x_1$. The ideal $J$ is prime by \cite[Example 3.4]{HU4}, so $x_1 \notin J$. Localize $S/J$ at $x_1$, and notice that \[JS_{x_1} = \mathfrak{a}S_{x_1}:(x_1,\dots,x_n)S_{x_1} = \mathfrak{a}S_{x_1}: S_{x_1} = \mathfrak{a}S_{x_1}\] and that $\mathfrak{a}S_{x_1}$ is generated by the set $\{u_{i,1}+\sum_{j=2}^n u_{i,j}\frac{x_j}{x_1} \; \vert \; 1 \leq i \leq s\}$. Thus, \[(S/J)_{x_1} \cong k[x_1,\dots,x_n][u_{i,j} \;\vert \; 1 \leq i \leq s, 2 \leq j \leq n][x_1^{-1}]\] is a regular ring. 
    
    Next, we must show that \[x_1(J^{[p]}:J) \not\subseteq \mathfrak{n}^{[p]}.\]
    
    Crucially, by Corollary \ref{corMain}, it is enough to show that $x_1(\mathfrak{b}^{[p]}:\mathfrak{b}) \not\subseteq \mathfrak{n}^{[p]}$ for some ideal $\mathfrak{b}$ contained in $J$ and generated by a regular sequence of length $\height(J)$. We choose $\mathfrak{b}$ to be generated by the sequence $\underline{\beta}$ as in Lemma \ref{lemmaRegularSecMax}. We are now done by Lemma \ref{lemmaNotInMax}.
    \end{proof}
    
We use the flat ascent of $F$-regularity to finish our proof. 

\begin{thm}\cite[Theorem 3.6]{Ab} \label{TheoremFlatAscent}
Let $(S,\mathfrak{n}) \longrightarrow (R,\mathfrak{m})$ be a faithfully flat extension of graded $F$-finite rings $S$ and $R$ of characteristic $p>0$ such that $S$ is strongly $F$-regular and the closed fiber $R/\mathfrak{n}R$ is Gorenstein and strongly $F$-regular. Then, $R$ is strongly $F$-regular.
\end{thm}

The above statement is proved in \cite{Ab} for the case where the rings $R$ and $S$ are local, though we use the graded version which immediately follows from applying the local statement to the closed fiber owing to the fact that the properties if being faithfully flat, Gorenstein and strongly $F$-regular are all local conditions and that they be checked by localizing at each maximal ideal. See \cite[Theorem 7.5, Remark 7.14]{MP} for details. 

\begin{thm}\label{theoremGenCU}
Let $R = K[x_1, \ldots, x_n]$ be a polynomial ring over a field $K$ and $I$ be an ideal of $R$ generated by a regular sequence of homogeneous elements. Then,
     \begin{enumerate}
         \item If $K$ has characteristic zero and $R/I$ has rational singularities, then the generic residual intersections of $R/I$ also have rational singularities.
         
         \item If $K$ is an $F$-finite field of positive characteristic and $R/I$ is $F$-rational, then the generic residual intersections of $R/I$ are strongly $F$-regular.
    \end{enumerate}
\end{thm}

\begin{proof}
It suffices to prove assertion $(2)$ of the theorem; since $F$-regular rings are $F$-rational, it then follows that $R/I$ is of $F$-rational type for $K$ of characteristic zero, and then by \cite[Theorem 4.3]{Karen} that $R/I$ has rational singularities.

Let $I = (f_1, \ldots, f_t)$ and $S$ be the polynomial ring $K[y_1,\ldots, y_t]$. Consider the homomorphism from $S$ to $R$ which maps $y_i$ to $f_i$ for $1 \leq i \leq t$. This map has closed fiber $R/I$ and it is faithfully flat since the ideal $I$ is generated by a regular sequence. This map induces, via a base change, a faithfully flat map from the generic residual intersection of the homogeneous maximal ideal of $S$ to the generic residual intersection of the ideal $I$ of $R$ since colon ideals commute with flat maps. Note that the generic residual intersection of $I$ is quasi-homogeneous by \cite[Example 3.4]{HU4} (and it is in fact standard graded if each $f_i$ is homogeneous of the same degree). This induced map on the generic residual intersections is thus a faithfully flat map of graded rings with closed fiber $R/I$. The closed fiber is a complete intersection $F$-rational ring by hypothesis and hence strongly $F$-regular by Theorem \ref{TheoremF-singularities} (5). Next, the generic residual intersection of $S$ is a (graded) strongly $F$-regular by Theorem \ref{theoremMaximalSFR}. The result now follows from Theorem \ref{TheoremFlatAscent}.
\end{proof}

\section{\texorpdfstring{$F$}{F}-Regularity of the Generic Links of Determinantal Rings} \label{Sec:LinkageSingularity}

Our final goal is to prove that the generic link of the generic determinantal ideal $I_t(X)$ of maximal minors is $F$-regular. Our strategy is to prove this by induction on $t$.

Note that the $t=1$ case is exactly the statement of Theorem \ref{theoremMaximalSFR} for generic links (that is, for $s=n$). In this section, we switch from the generic residual intersection of the $t=1$ case (which we dealt with in 
\S\ref{Sec:ResIntSingularity}) to the generic link of the $t \geq 2$ case. This is because the generic residual intersections of $I_t(X)$ for $t \geq 2$ are \emph{not} Cohen-Macaulay unless they are also links; see \cite[Theorem 1.1]{EU}. Thus, for example, they may never be $F$-regular.  

\begin{rem}\label{remarkGenricFixGens}
By Lemma \ref{lemResIntEquiv} and Lemma \ref{lemmaAnyGensOkay}, $F$-regularity and all other relevant properties are independent of the choice of a generating set of $\mathfrak{m}$. So, for the rest of this section, we fix the generators of $I_t(X)$ to be the maximal minors of $X$, $\Delta_{1},\dots,\Delta_r$, where $r = \binom{n}{t}$. 

Let $U=(u_{i,j})$ be an $(n-t+1) \times r$ matrix of indeterminates and $S = R[U]$ and let $\mathfrak{a} \subseteq IS$ be the ideal generated by the entries of $U[\Delta_{1} \dots \Delta_r]^T$. We fix $J = \mathfrak{a}:IS$ to be our choice of a generic link. 
\end{rem}

\begin{lem}\label{lemmaNotInGeneric}
    Let $\mathfrak{a}$ and $R[U]$ be defined as in Remark \ref{remarkGenricFixGens}, and let $\mathfrak{m}$ be the homogeneous maximal ideal of $R[U]$, then $x_{1,n}(\mathfrak{a}^{[p]}:\mathfrak{a}) \not\subseteq \mathfrak{m}^{[p]}$.  
\end{lem}

\begin{proof}
    Let $a_i = u_{i,1}\Delta_1 + \dots u_{i,r}\Delta_r$ for $1 \leq i \leq n-t+1$, and let $k_1, \dots, k_{n-t+1}$ be the indices such that
    $\Delta_{k_i} = [i,t+i-1]$. Notice that $x_{1,n}(a_1\dots a_{n-t+1})^{p-1} \in  x_{1,n}(\mathfrak{a}^{[p]}:\mathfrak{a})$. 
    
    Since $\mathfrak{m}^{[p]}$ is a monomial ideal, a polynomial $f \in \mathfrak{m}^{[p]}$ if and only if every monomial term of $f$ is in $m^{[p]}$. So it is enough to show that for some monomial ordering $<$, we have \[\initial_<\big(x_{1,n}(a_1\dots a_{n-t+1})^{p-1}\big) \notin m^{[p]}.\] 
    
    Use the lexicographical order induced from the variable order \[u_{1,k_1} > u_{2,k_2} > \dots > u_{n-t+1,k_{n-t+1}} > \text{ the remaining } u_{i,j} >\] \[x_{1,1} > x_{1,2} > \dots > x_{1,n} > x_{2,1} > \dots > x_{2,n} > \dots > x_{t,n}.\] For a polynomial $f$, let $\initial_<(f)$ represent the initial monomial of $f$ with respect to this order. By expanding the determinant along the first row, we get $\initial_<([i,t+i-1]) = x_{1,i}x_{2,i+1}\dots x_{t,t+i-1}$, which is just the product along the main diagonal of $[i,t+i-1]$. Thus,
     \begin{eqnarray*}
        \initial_<\big(x_{1,n}(a_1 \dots a_{n-t+1})^{p-1}) &=& x_{1,n}(\initial_<(a_1)\dots \initial_<\big(a_{n-t+1})\big)^{p-1}\\
        &=& x_{1,n}(u_{1,k_1}\initial_<([1,t])\dots u_{n-t+1,k_{n-t+1}}\initial_<([n-t+1,n])\big)^{p-1}\\
        &=& x_{1,n}\Bigg(\prod_{i = 1}^{n-t+1}u_{i,k_{i}}^{p-1}\Bigg)\Bigg(\prod_{0 \leq j-i \leq n-t} x_{i,j}^{p-1}\Bigg).
    \end{eqnarray*}

    We have $\initial_<\big(x_{1,n}(a_1\dots a_{n-t+1})^{p-1}\big) \notin m^{[p]}$, so we are done.
\end{proof}

We are now ready to prove the main result of this paper.

\begin{thm}\label{theoremGenericSFR}
    Let $X = (x_{i,j})$ be a $t \times n$ matrix of indeterminates for $n \geq t$, $K$ a field, and $R = K[X]$. Let $I_t(X)$ denote the ideal of $R$ generated by the size $t$ minors of $X$. Then, 
     \begin{enumerate}
         \item If $K$ has characteristic zero, the generic link of $R/I_t(X)$ have rational singularities.
         
         \item If $K$ is an $F$-finite field of positive characteristic, the generic link of $R/I_t(X)$ is strongly $F$-regular.
     \end{enumerate}
\end{thm}

\begin{proof}
    It suffices to prove assertion $(2)$ of the theorem; since $F$-regular rings are $F$-rational, it then follows that $R/I$ is of $F$-rational type for $K$ of characteristic zero, and then by \cite[Theorem 4.3]{Karen} that $R/I$ has rational singularities.
    
    We proceed by induction on $t$. The case $t = 1$ follows directly from Theorem \ref{theoremMaximalSFR}. Now we assume the statement for $t-1 \geq 1$ and show that it is true for $t$. To do this, we apply Glassbrenner's criterion (Theorem \ref{theoremGlassbrenner}) with $s = x_{1,n}$. Note that $J$ is prime by Proposition \ref{propLinksPrime}, so it follows that $x_{1,n} \notin J$.

    In what follows, we prove rigorously that the localization of the generic link of the ring defined by the ideal $I_t(X_{t \times n})$ at $x_{1,n}$ gives a smooth extension of the generic link of the ring defined by $I_{t-1}(X_{(t-1) \times (n-1)})$. This checks condition $(1)$ of Glassbrenner's criterion.
    
    Let $r' = \binom{n-1}{t-1}$. Let $J$ be defined as in Remark \ref{remarkGenricFixGens}. Without loss of generality, assume that $\Delta_{1}, \dots, \Delta_{r'}$ are the maximal minors of $X$ that contain the last column. 

    Set $S = R[U]$ and localize $S/J$ at $s = x_{1,n}$. Let $\Delta_{j}' = x_{1,n}^{-1}\Delta_{j}$ for $1 \leq j \leq r'$. Notice that $\Delta_1',\dots, \Delta_{r'}'$ are the maximal minors of the $(t-1) \times (n-1)$ matrix $X'$ with entries $\{x_{i,j}' \;\vert\; 2 \leq i \leq t; 1 \leq j \leq n-1 \}$ where $x_{i,j}' = x_{i,j} -  x_{1,j} x_{i,n} x_{1,n}^{-1}$. Furthermore, for $$\{i_1,\dots,i_t\} \subseteq \{1,\dots,n-1\},$$ we have \[\det \begin{bmatrix}
        x_{1,i_1} &\dots& x_{1,i_t}\\
        x_{2,i_1}'&\dots& x_{2,i_t}'\\
        \vdots &  & \vdots\\
        x_{t,i_1}'&\dots& x_{t,i_t}'
    \end{bmatrix} = \det \begin{bmatrix}
        x_{1,i_1} &\dots& x_{1,i_t}\\
        x_{2,i_1}&\dots& x_{2,i_t}\\
        \vdots &  & \vdots\\
        x_{t,i_t}&\dots& x_{t,i_t}
    \end{bmatrix}, \] since determinants are invariant under row operations.

    Thus, for $l > r'$, $\Delta_{l}$ can be written as $\sum_{k=1}^{r'}f_{l,k}\Delta_k'$ with $f_{l,k} \in \{x_{1,j} \; \vert \; 1 \leq j < n\}$.

    Notice that $a_i = \sum_{j=1}^{r} u_{i,j}\Delta_{j}$ for $1 \leq i \leq n-t+1$ is a generating set of $\mathfrak{a}$, which gives us that $a_i' = x_{1,n}^{-1}a_i$ for $1 \leq i \leq n-t+1$ is a generating set of $\mathfrak{a}S_{x_{1,n}}$. Now we have that: 

    \begin{eqnarray*}
        a_i' &=& x_{1,n}^{-1}\sum_{j=1}^{r} u_{i,j}\Delta_{j}\\
             &=& \sum_{j=1}^{r'} u_{i,j}\Delta_{j}'+\sum_{l=r'+1}^r x_{1,n}^{-1}u_{i,l}\Delta_{l}\\
             &=& \sum_{j=1}^{r'} u_{i,j}\Delta_{j}'+\sum_{l=r'+1}^r x_{1,n}^{-1}u_{i,l}\sum_{j=1}^{r'}f_{l,j}\Delta_j'\\
             &=& \sum_{j=1}^{r'} u_{i,j}\Delta_{j}'+\sum_{j=1}^{r'}\Bigg(\sum_{l=r'+1}^r x_{1,n}^{-1}u_{i,l}f_{l,j}\Bigg)\Delta_j'\\
             &=& \sum_{j=1}^{r'} \big(u_{i,j} + \sum_{l=r'+1}^r x_{1,n}^{-1}u_{i,l}f_{l,j}\big)\Delta_j' \\
             &=& \sum_{j=1}^{r'} u_{i,j}'\Delta_{j}',
    \end{eqnarray*}

    where \[u_{i,j}' = u_{i,j} + \sum_{l=r'+1}^r x_{1,n}^{-1}u_{i,l}f_{l,j} \quad \text{for} \quad 1 \leq j\leq r'.\] Notice that, for $1\leq i\leq n+t-1$ and $1 \leq j\leq r'$, $u_{i,j}$ is not a factor of any monomial term of \[\sum_{l=r'+1}^r x_{1,n}^{-1}u_{i,l}f_{l,j},\] so the set \[\{u_{i,j}' \; \vert \; 1\leq i\leq n+t-1 , 1 \leq j\leq r'\}\] is algebraically independent. 

    Let $W = (w_{i,j})$ be an $(n-t+1) \times r'$ matrix of intdeterminates, and let $Y = (y_{i,j})$ be a $(t-1) \times (n-1)$ matrix of indeterminates. Then, by Lemma \ref{lemmaInvertx11}, we have an isomorphism \[\varphi: (K[X][U])_{x_{1,n}}\rightarrow (K[Y][W])[x_{i,j} \; \vert \; j = n \text{ or } i = 1][u_{i,j} \; \vert \; 1\leq i\leq n+t-1, r'<j\leq r][x_{1,n}^{-1}],\] where $\varphi(\Delta_{i}')=\delta_i$ for $1 \leq i \leq r'$ and $\{\delta_i\; \vert \; 1 \leq i \leq r'\}$, the set of the maximal minors of $I_{t-1}(Y)$. Furthermore $\varphi(u_{i,j}') = w_{i,j}$ for $1 \leq j\leq r'$. This induces an isomorphism 
    \[(K[X][U]/J)_{x_{1,n}} \cong (K[Y][W]/(\mathfrak{b}:I_{t-1}(Y))[x_{i,j} \; \vert \; j = n \text{ or } i = 1][u_{i,j} \; \vert \; 1\leq i\leq n+t-1, r'<j\leq r][x_{1,n}^{-1}],\]
    where $\mathfrak{b}$ is the ideal generated by $W[\delta_i,\dots,\delta_{r'}]^T$. This proves that $(S/J)_{x_{1,n}}$ is (strongly) $F$-regular by induction.
     
     Next, to check condition $(2)$ of Glassbrenner's criterion, we must show that \[x_{1,n}(J^{[p]}:J) \not\subseteq \mathfrak{m}^{[p]}.\]  Crucially, by Lemma \ref{lemmaMain}, it is enough to show that $x_{1,n}(\mathfrak{a}^{[p]}:\mathfrak{a}) \not\subseteq m^{[p]}$, where the ideal $\mathfrak{a}$ is as in Remark \ref{remarkGenricFixGens}. We are now done by Lemma \ref{lemmaNotInGeneric}.
    \end{proof}

\begin{rem}
    Notice that the $F$-purity of generic links of generic determinantal rings defined by maximal minors directly follows from applying Lemma \ref{lemmaMain} along with Lemma \ref{lemmaNotInGeneric}. 
\end{rem}

\section*{Acknowledgments}
The authors would like to thank Linquan Ma for introducing us to questions on singularities of generic links. We also thank him for his patient guidance and careful reading of the manuscript. The authors also gratefully acknowledge Bernd Ulrich for several helpful discussions and for suggesting an improvement in one of our results. We thank the anonymous referee for their detailed comments and suggested corrections.

\end{document}